\documentclass{ifacconf}

\usepackage{float} 
\usepackage{graphicx}      
\usepackage{natbib}        
\usepackage{amsmath}       
\usepackage{amssymb}
\usepackage{booktabs}
\usepackage{multirow}
\usepackage{subcaption}
\usepackage{verbatim}
\usepackage{varioref}
\usepackage{diagbox}
\usepackage{hhline}
\usepackage{comment}
\usepackage{afterpage}

\usepackage{caption}
\usepackage{tikz}
\usepackage{calc} 
\usepackage{siunitx}
\usepackage{bm}
\usepackage{makecell}

\newtheorem{theorem}{Theorem}

\makeatletter
\newcommand{\CaptionWidthFallback}{%
  \ifdim\linewidth>0pt \linewidth \else \textwidth \fi
}
\makeatother

\begin{document}
\begin{frontmatter}

\title{Verifiable Model-Free Safety Filters via Reinforcement Learning}


\author[1]{Bihui Yin} 
\author[1]{Yiwen Lu} 
\author[2]{Yuchen Jiang}
\author[1]{Yilin Mo}

\address[1]{Department of Automation and BNRist, Tsinghua University (e-mail: yinbh25@mails.tsinghua.edu.cn, luyw20@tsinghua.org.cn, ylmo@tsinghua.edu.cn)}
\address[2]{Control and Simulation Center, National Key Laboratory of Complex System Control and Intelligent Agent Cooperation, Harbin Institute of Technology (e-mail: yc.jiang@hit.edu.cn)}

\begin{abstract}

    This paper presents a reinforcement learning approach of a model-free safety filter, 
    drawing inspiration from the framework of model-based Predictive Safety Filters (PSFs). 
    Similar to conventional PSFs, our method adopts a Quadratic Programming (QP) formulation 
    by representing the filter as an unrolled QP solver network. 
    However, unlike existing PSFs that derive QP parameters explicitly from system models, 
    we learn these parameters directly through Deep Reinforcement Learning (DRL), 
    thereby eliminating the dependency on accurate system identification. 
    Furthermore, compared to traditional neural network-based methods, 
    this QP structure allows us to furnish a formal certificate for the persistent safety of the learned filter. 
    Numerical results demonstrate that our method outperforms 
    both conventional model-based PSFs and RL-trained Multi-Layer Perceptron (MLP) baselines 
    in terms of safety guarantees, minimal intervention, and per-step computational load.     

 \end{abstract}


\end{frontmatter}

\section{Introduction}
The rapid advancement of modern control systems has expanded their applications into numerous domains, 
including industrial automation, transportation, and healthcare. 
However, ensuring safety in these systems remains a fundamental challenge \citep{Ames2017}, 
as inaccurate system models, environmental uncertainties, and aggressive control policies can lead to 
constraint violations that may cause catastrophic failures. 

Safety filters \citep{Hsu2023} have emerged as a modular solution to this challenge, 
modifying unsafe control inputs to ensure persistent constraint satisfaction with minimal intervention. 
As shown in Fig. \ref{fig:safety_filter}, 
the safety filter is placed between the controller and the physical system. 
Its role is to safeguard the system against the proposed control action $\hat{u}$, 
which may originate from a controller that is untrusted, aggressive, or whose safety characteristics are unverified. 
Applying such an input directly could lead to constraint violations. 
To prevent this, the filter first evaluates the safety of
proposed control action $\hat{u}$ based on current state $x$.  
When $\hat{u}$ is safe, the filter passes it through unmodified ($u = \hat{u}$);
otherwise, it computes a minimally modified safe alternative $u$ that satisfies all constraints. 
With the safety filter, we turn the constrained dynamical system into an unconstrained safe system \citep{Wabersich2021}. 

\begin{figure}[h]
   \begin{center}
   \includegraphics[width=8.4cm]{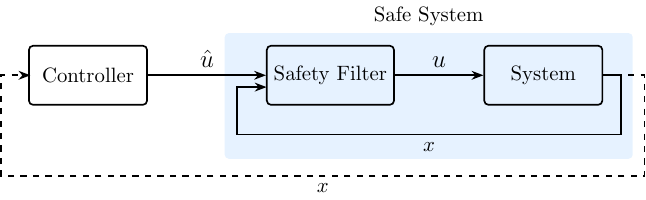}    
   \caption{Safety filter integrated in control architecture, 
   situated between an open-loop/closed-loop controller and the real system to enforce safety. 
   It validates the proposed input $\hat{u}$ and computes a minimally invasive correction $u$ to guarantee constraint satisfaction.}
   \label{fig:safety_filter}
   \end{center}
\end{figure}

Several works have developed model-based safety filters that require accurate knowledge of system dynamics.
\textit{Hamilton-Jacobi (HJ) reachability analysis} provide formal safety guarantees by computing reachable sets through partial differential equations \citep{Margellos2011}.
While effective for low-dimensional systems, 
HJ methods exhibit exponential computational complexity in high dimensions \citep{Herbert2021}, 
which persists even with GPU acceleration \citep{Long2024}.
\textit{Control Barrier Functions (CBFs)}\citep{Wieland2007, Oswin2024, Mestres2025} and their extensions \citep{Agrawal2017, Cosner2022} 
convert safety constraints into forward invariance problems, avoiding explicit trajectory computation.
However, constructing valid CBFs requires significant manual effort and restrictive assumptions. 

Another popular model-based alternative is \textit{Predictive Safety Filters (PSFs)}, 
which formulate safety enforcement as real-time optimization problems similar to Model Predictive Control (MPC) \citep{Wabersich2018, Wabersich2021}. 
Though effective, PSFs require solving optimization problems at each time step through iterative numerical methods, 
resulting in substantial and unpredictable computational load that hinder real-time deployment \citep{Viljoen2024}.
While these model-based approaches can provide formal safety guarantees, 
they face practical limitations including dependency on accurate system models and computational complexity.

In contrast, \textit{data-driven safety filters} avoid reliance on explicit system models by learning from data \citep{Jason2025, Tang2024}.
Among these, some efforts use Deep Reinforcement Learning (DRL) 
to approximate Hamilton-Jacobi-Bellman solutions \citep{Bansal2021} 
or learn safety certificates \citep{Lavanakul2024}. 
Other lines of work focus on learning Control Barrier Functions (CBFs), 
either from demonstrations \citep{Robey2020} or sampled state data \citep{Dawson2022}. 
Meanwhile, shielding methods combine learned controllers with model-based monitors \citep{Bastani2021}, 
but inherit the computational cost of online verification. 
Although effective in many settings, these data-driven methods typically entail trade-offs 
between formal safety-verification capability and computational efficiency \citep{Hsu2023}.

To integrate the complementary strengths of model-based structure and data-driven adaptation,
we propose a QP-structured safety filter and employ DRL to learn its parameters. 
Inspired by a recent paradigm that frames MPC controllers as learnable QP problems \citep{Lu2023}, 
we parameterize the PSF as an unrolled QP solver network. 
Crucially, instead of deriving the QP parameters from an explicit system model, 
we treat them as learnable parameters and optimize them directly via DRL. 
This innovative approach allows for end-to-end training of a filter 
that is computationally efficient and model-free, 
while preserving the interpretable and verifiable structure inherent to conventional PSFs. 

\textit{Contributions:} This work presents an approach of constructing a model-free safety filter 
by parameterizing the filter as an unrolled QP solver network with parameters learned via DRL. 
This approach eliminates the model dependency and computational cost 
while preserving an interpretable QP structure that enables formal verification. 
Specifically, we provide a safety certificate for the learned filter under linear system dynamics. 
Furthermore, extensive benchmark experiments are conducted, which confirmed 
the superior performance and computational efficiency of our method 
compared to both conventional PSFs and DRL-trained Multi-Layer Perceptron (MLP) baselines. 

The rest of the paper is organized as follows: 
Section 2 formulates the PSF problem for nominal linear systems and establishes its QP representation. 
Section 3 details our learning-based approach, including the unrolled QP architecture, 
the reinforcement learning framework for parameter optimization, and the design of reward function tailored for safety filtering.
Section 4 provides theoretical analysis, establishing a formal certificate for the persistent safety of the learned filter.
Section 5 reports experimental results on stabilization and tracking tasks across benchmark systems, 
with comparative analysis against baseline methods. 
Conclusions and future work are discussed in the final section.

\textit{Notations:}
$\mathbb{R}^n$ denotes the $n$-dimensional Euclidean space. 
$\mathbb{S}_{+}^n$ and $\mathbb{S}_{++}^n$ denote the sets of symmetric positive semi-definite and definite $n \times n$ matrices, respectively.
Subscripts indicate time indices (e.g., $x_k$ denotes system state at step $k$).
The Kronecker product is denoted by $\otimes$, and $\text{diag}(A_1,\dots,A_n)$ represents a block diagonal matrix. 
The set of integers in the interval $[a,b] \subset \mathbb{R}$ is $\mathcal{I}_{[a,b]}$.

\section{Problem Formulation and Preliminaries}
\label{sec:problem_formulation}
This section formulates the conventional Predictive Safety Filter (PSF) problem 
and establishes its quadratic programming (QP) representation. 
We begin by presenting the optimization-based formulation of PSF for nominal linear systems, 
then detail the transformation of this constrained optimization problem into a standard QP form. 
This formulation lays the foundation for the learning-based approach developed in the subsequent section.

\subsection{PSF for Nominal Linear Systems}
\label{subsec:psf_linear}
The Predictive Safety Filter (PSF) acts as a supervisory module that modifies an arbitrary (and potentially unsafe) 
reference control input $\hat{u}$ to ensure system safety. 
For a discrete-time linear time-invariant (LTI) system, at each time step, 
given the current state $x_0$ and reference input $\hat{u}$, 
the PSF solves the following optimization problem:
\begin{subequations}
\label{eq:psf_problem}
\begin{align}
      \min_{u_{0:N-1}} \quad & \| \hat{u} - u_0 \|^2 \label{eq:psf_obj}, \\
      \text{s.t.} \quad & x_{k+1} = A x_k + B u_k, \quad \forall k \in \mathcal{I}_{[0, N-1]} \label{eq:psf_dynamics}, \\
      & S_X \, x_{k+1} \ge d_X, \quad \forall k \in \mathcal{I}_{[0, N-1]} \label{eq:x_polytope_constraints}, \\
      & S_U \, u_k \ge d_U, \quad \forall k \in \mathcal{I}_{[0, N-1]} \label{eq:u_polytope_constraints}, \\
      & x_N \in \mathcal{X}_f \label{eq:psf_terminal_set},
\end{align}
\end{subequations}

where $N$ is the prediction horizon. The components of this optimization problem are defined as follows:

- \textit{Objective function} \eqref{eq:psf_obj}: Minimizes the squared deviation 
between the reference input $\hat{u}$ and the filtered control $u_0$, ensuring minimal intervention. 

- \textit{System dynamics} \eqref{eq:psf_dynamics}: Enforces the discrete-time linear time-invariant 
system dynamics throughout the prediction horizon, 
where $x_k \in \mathbb{R}^{n_{\text{sys}}}$ is the state vector, 
$u_k \in \mathbb{R}^{m_{\text{sys}}}$ is the control input, 
and $A \in \mathbb{R}^{n_{\text{sys}} \times n_{\text{sys}}}$, 
$B \in \mathbb{R}^{n_{\text{sys}} \times m_{\text{sys}}}$ are system matrices. 

- \textit{State safe constraints} \eqref{eq:x_polytope_constraints} 
and \textit{Input safe constraints} \eqref{eq:u_polytope_constraints}: 
Ensures all future states and control inputs remain within the safe polytopes. 
$S_X \in \mathbb{R}^{m_X \times n_{\text{sys}}}$, $d_X \in \mathbb{R}^{m_X}$, 
and $S_U \in \mathbb{R}^{m_U \times m_{\text{sys}}}$ and $d_U \in \mathbb{R}^{m_U}$ 
define the linear inequalities. 

- \textit{Terminal constraint} \eqref{eq:psf_terminal_set}: 
Requires the final predicted state $x_N$ to lie within the terminal safe set $\mathcal{X}_f$, 
which ensures persistent feasibility and stability guarantees. 
We assume the terminal set is a polyhedron 
$\mathcal{X}_f = \{x \in \mathbb{R}^{n_{\text{sys}}} \mid F x \geq g\}$, 
where $F \in \mathbb{R}^{m_F \times n_{\text{sys}}}$ and $g \in \mathbb{R}^{m_F}$. 
This polyhedral representation facilitates subsequent transformation and computation 
within the optimization framework. 

The first control input $u_0^*$ of the optimal solution will be applied to the system. 
This formulation ensures safety by enforcing state and input constraints over the prediction horizon 
while minimally modifying $\hat{u}$, with the terminal set constraint providing long-term safety guarantees.

\subsection{Standard QP Representation}
\label{subsec:qp_representation}
We can transform the original PSF problem \eqref{eq:psf_problem} 
into a standard QP formulation: 
\begin{subequations}
\label{eq:standard_qp}
\begin{align}
    \min_{y} \quad & \tfrac{1}{2}y^\top P y + q^\top y,\\
    \text{s.t.} \quad & H y + b \ge 0,
\end{align}
\end{subequations}
where $y = [u_0^\top,\; u_1^\top,\; \dots,\; u_{N-1}^\top]^\top \in\mathbb{R}^{n_{\text{qp}}}$, 
$P\in\mathbb{S}_+^{n_{\text{qp}}}$, 
$q\in\mathbb{R}^{n_{\text{qp}}}$, $H\in\mathbb{R}^{m_{\text{qp}}\times n_{\text{qp}}}$, 
$b\in\mathbb{R}^{m_{\text{qp}}}$, 
and the problem dimensions are given by:
\begin{subequations}
\label{eq:qp_dimensions}
\begin{align}
    n_{\text{qp}} &= N m_{\text{sys}} \label{eq:qp_dim_n}, \\
    m_{\text{qp}} &= N (m_X + m_U) + m_F \label{eq:qp_dim_m},
\end{align}
\end{subequations}
where $m_X$ is the number of rows in $S_X$, $m_U$ is the number of rows in $S_U$, 
and $m_F$ is the number of terminal constraints.

A common construction produces:
\begin{align}
    P &= \mathrm{diag}\big(I_{m_{\text{sys}}}, \varepsilon I_{(N-1)m_{\text{sys}}}\big),\label{eq:P_ridge}\\
    q &= -E_1^\top \hat{u},\label{eq:q_definition}
\end{align}
where $E_1=[I_{m_{\text{sys}}}\; \mathbf{0}]$ selects $u_0$, 
and we add a small ridge $\varepsilon>0$ in \eqref{eq:P_ridge} 
for numerical stability (this makes $P\in\mathbb{S}_{++}$). 

The inequality constraints encode state and input box constraints 
over the horizon and the terminal constraint. 
According to system models and given safety constraints, 
the QP parameters $(H, b)$ can be constructed as:
\begin{subequations}
\label{eq:qp_parameters}
\begin{align}
    H &= \begin{bmatrix}
        (I_N \otimes S_X) \mathcal{B} \\
        (I_N \otimes S_U) \\
        F \mathcal{B}_N
    \end{bmatrix}, \label{eq:qp_param_H} \\
    b &= \begin{bmatrix}
        - \mathbf{1}_N \otimes d_X + (I_N \otimes S_X) \mathcal{A} x_0 \\
        - \mathbf{1}_N \otimes d_U \\
        g - F \mathcal{A}_N x_0
    \end{bmatrix}, \label{eq:qp_param_b}
\end{align}
\end{subequations}

with the auxiliary matrices and vectors:
\begin{equation}
\label{eq:qp_aux}
\begin{aligned}
    \mathcal{A} &= \begin{bmatrix} A \\ A^2 \\ \vdots \\ A^N \end{bmatrix}, \hspace{1em}
    \mathcal{B} = \begin{bmatrix} 
        B &  &  \\
        AB & B &  \\
        \vdots & \ddots & \ddots \\
        A^{N-1}B & \cdots & AB & B 
    \end{bmatrix}, \\
    \mathcal{A}_N &= \text{last } n_{\text{sys}} \text{ rows of } \mathcal{A}, 
    \mathcal{B}_N = \text{last } n_{\text{sys}} \text{ rows of } \mathcal{B}.
\end{aligned}
\end{equation}

Although the standard QP representation enables the use of off-the-shelf solvers, 
it inherently relies on accurate system models $(A, B)$ 
and incurs substantial computational overhead due to iterative online optimization.
These limitations motivate the learning-based approach developed in the next section, 
where we parameterize the QP solver as a neural network and learn its parameters through reinforcement learning.

\section{Learning Model-Free Safety Filters}
\label{sec:proposed_method}
This section details the proposed learning-based approach for our model-free safety filters.  
First, we introduce the safety filter network architecture and its learnable parameters. 
Then, the reinforcement learning framework for training the filter policy will be presented.

\subsection{Network Architecture and Learnable Parameters}
\label{subsec:network_architecture}

According to \ref{subsec:qp_representation}, for the QP parameters $(P, q, H, b)$,  
$P$ and $q$ are known under the given control input $\hat{u}$, 
whereas $H$ and $b$ rely heavily on system models. 
Instead of explicitly constructing $H$ and $b$ from system dynamics, 
which may be inaccurate or unavailable, 
we treat them as \textit{learnable} parameters within our network architecture.
Especially, \eqref{eq:qp_param_H} reveals that matrix $H$ remains \textit{invariant} 
across different initial states and reference control inputs $(x_0, \hat{u})$, 
i.e., only a single matrix $H$ need to be learned for a specific system.
For vector $b$, \eqref{eq:qp_param_b} motivates its parameterization as a 
\textit{state-dependent affine transformation}:
\begin{equation}
b(x_0; W_b, b_b) = W_b x_0 + b_b, 
\label{eq:b_affine_transform}
\end{equation}
where $W_b$ and $b_b$ are learnable matrix and vector of proper dimensions.

Therefore, the original QP problem \eqref{eq:standard_qp} is transformed into a parameterized form: 
\begin{subequations}
\label{eq:parameterized_qp}
\begin{align}
    \min_{y} \quad & \tfrac{1}{2}y^\top P y + q^\top y,\\
    \text{s.t.} \quad & H y + W_b x_0 + b_b \ge 0,
\end{align}
\end{subequations}
where $y\in\mathbb{R}^{n_{\text{qp}}}$ is the optimization variable, 
$P\in\mathbb{S}_+^{n_{\text{qp}}}$ and $q\in\mathbb{R}^{n_{\text{qp}}}$ are known constants from \eqref{eq:P_ridge} and \eqref{eq:q_definition}, 
while $H\in\mathbb{R}^{m_{\text{qp}}\times n_{\text{qp}}}$, 
$W_b \in \mathbb{R}^{m_{\text{qp}} \times n_{\text{sys}}}$ and $b_b \in \mathbb{R}^{m_{\text{qp}}}$ are learnable parameters.  

To compute solutions and backpropagate gradients for the learned QP problem \eqref{eq:parameterized_qp}, 
we adopt the unrolling technique \citep{Monga2021} by executing a fixed number($n_{\text{iter}}$) of QP solving iterations  
and differentiates through the entire computational path, treating the solver as a recurrent computational graph.
This avoids the need for implicit differentiation, which requires solving to optimality before differentiation.

Specifically, we employ the Primal-Dual Hybrid Gradient (PDHG) algorithm \citep{Chambolle2011} 
as the underlying QP solver, whose iterations can be expressed as:
\begin{equation}
\label{eq:pdhg_iterations}
\begin{aligned}
z^{i+1} &= \Pi_{\mathbb{R}_+^{m_{\text{qp}}}} \left( (I - 2\alpha F) z^i + \alpha (I - 2F)\lambda^i - 2\alpha\mu \right), \\
\lambda^{i+1} &= F\left(z^i + \lambda^i\right) + \mu,
\end{aligned}
\end{equation}
where $z = H y + b$ is the primal variable, $\lambda$ is the dual variable, 
$\alpha > 0$ is the step size, and $F = (I + H P^{-1} H^\top)^{-1}$, 
$\mu = F(H P^{-1} q - b)$ are parameters derived from the QP problem. 
The core operations of this PDHG iterations 
lies in decomposing each iteration into linear affine operations, whose parameters $(F, \mu)$ depend on the QP parameters $(P, q, H, b)$,
and metric projections applied to optimization variables which is equivalent to the ReLU activation in neural networks \citep{Lu2023}.

In our \textit{unrolled QP solver}, 
gradients are computed directly from the intermediate solution after exactly $n_{\text{iter}}$ iterations, 
maintaining validity even before convergence. 
Empirically, a small number of iterations (e.g., $n_{\text{iter}} = 10$) suffices 
to achieve satisfactory safety filter performance while significantly reducing the computational burden of unrolling. 
This efficiency stems from our model-free formulation, 
which enables learning a QP problem that is not only performant but also easy to solve.

Fig. \ref{fig:network_architecture} illustrates our proposed safety filter network architecture, 
which replaces the online QP solver with a parameterized unrolled QP network that mimics QP solving iterations.
Within this structure, the filter policy——which maps initial states $x_0$ and given control input $\hat{u}$ 
to safe control actions $u_0$ is implemented through a fixed-count iterative solver,
whose parameters can subsequently be optimized using DRL methodologies.
\begin{figure}[h]
   \begin{center}
  \includegraphics[width=8.4cm]{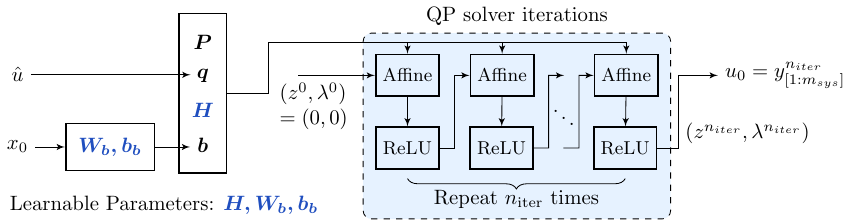} 
  \caption{Proposed safety filter policy architecture, 
    which solves a QP problem in form \eqref{eq:parameterized_qp} via an unrolled QP solver network, 
    whose structure mimics $n_{\text{iter}}$ PDHG iterations to approximate QP solutions.
    The learnable parameters are $H, W_b, b_b$.}
    \label{fig:network_architecture}
   \end{center}
\end{figure}

\subsection{Reinforcement Learning Setup}
\label{subsec:rl_setup}
The reinforcement learning (RL) framework is defined by five key components: 
agent, environment, state, action, and reward \citep{sutton2018reinforcement}. 
In our context, the \textit{agent} represents the learned safety filter, 
the \textit{environment} encompasses the system dynamics (excluding the filter) and an arbitrary external controller, 
the \textit{state} consists of the system state $x$ and reference control input $\hat{u}$, 
and the \textit{action} is the filter's output $u_0$. 
The \textit{reward} function is designed to guide the agent toward effective safety filtering, 
balancing safety enforcement and minimal intervention.

We employ Proximal Policy Optimization (PPO) \citep{Schulman2017} for training, 
with a reward function structured as: 
\begin{equation}
r_t = r_{\text{safety}} + r_{\text{deviation}} + r_{\text{min\_interv}} + r_{\text{survival}} + r_{\text{termination}}
\end{equation}
The components are designed as follows:

\textit{1. Safety Violation Penalty:} 
A fixed penalty $-k_1$ is applied when state or input constraints are violated, 
regardless of magnitude. The violation indicator $\phi_t^{\text{viol}}$ is defined as:
\begin{equation}
  \phi_t^{\text{viol}} = \begin{cases} 
  1 & \begin{aligned} &\text{if\,\,\,\,} \exists i \text{ such that } (S_X x_t)_{[i]} < (d_X)_{[i]} \\ 
  & \text{or } \exists j \text{ such that } (S_U u_t)_{[j]} < (d_U)_{[j]} \end{aligned} \\
  0 & \text{otherwise}
  \end{cases}
\end{equation}
Thus, $r_{\text{safety}} = -k_1 \phi_t^{\text{viol}}$ with $k_1 > 0$.

\textit{2. Control Deviation Penalty:} 
To minimize unnecessary intervention, we penalize the squared deviation from the reference input: 
$r_{\text{deviation}} = -k_2 \| \hat{u} - u_0 \|^2$, where $k_2 > 0$ tunes the trade-off between safety and fidelity.

\textit{3. Minimal-Intervention Reward:}
To encourage zero intervention when safe, 
a positive reward $k_3$ is granted if the control modification is below a threshold 
$\delta_{\min} > 0$: $r_{\text{min\_interv}} = k_3 \cdot \mathbb{I}(\| \hat{u} - u_0 \| < \delta_{\min})$, 
where $\mathbb{I}(\cdot)$ is the indicator function and $k_3 > 0$.

\textit{4. Auxiliary Rewards:}
To mitigate reward sparsity issues and prevent the agent from prematurely terminating episodes 
to avoid future penalties ``suicide'' behaviors), 
we include two supplementary rewards:  \\
- \textit{Time-step survival bonus:} 
$r_{\text{survival}} = k_4$ (awarded at every time step for sustained operation) \\
- \textit{Early termination penalty:} 
$r_{\text{termination}} = -k_5$ (applied only when episode terminates due to unbearable constraint violation)

This reward structure provides dense learning signals and promotes long-term safety and performance.

\section{Certificate for Persistent Safety}
\label{sec:theoretical_analysis}
While our proposed safety filter is data-driven and model-free, 
its underlying QP structure allows for formal verification of its safety properties post-training. 
In this section, we provide a formal certificate for the \textit{persistent safety} of our learned safety filter. 
Persistent safety ensures that if the system starts in a safe state $x_0 \in \mathcal{X}_0$, 
the filter will keep the next state inside the safety bounds at every step. 
Consequently, it can always produce an action that keeps the system 
within the prescribed safety constraints at all subsequent steps.

Let the safety filter policy be denoted by $\pi_\theta(x_0, \hat{u})$, 
where $\theta = \{H, W_b, b_b\}$ represents the learned parameters. 
The policy outputs a control action $u_0 = y^*_{[1:m_{\text{sys}}]}$. 
Throughout this section, we assume that the optimal solution 
of the learned QP problem is attained, 
which can be ensured by allowing the QP solver to 
run sufficient iterations until convergence when deploying.
Thus, $y^*$ is exactly the optimal solution to the QP problem:
\begin{equation}
\label{eq:learned_qp}
y^* \in \arg\min_{y} \left\{ \tfrac{1}{2}y^\top P y + q^\top y \mid Hy + b(x_0) \ge 0 \right\},
\end{equation}
with $q = -E_1^\top \hat{u}$ and $b(x_0) = W_b x_0 + b_b$.
Additionally, without loss of generality, we assume the safe set $\mathcal{X}_0$
can be expressed as a polyhedral form: $\mathcal{X}_0 = \{x \mid Gx \le c\}$.

To verify the persistent safety of proposed safety filter, 
a sufficient certificate is provided as follows:

\begin{theorem}[Certificate for Persistent Safety]
\label{thm:persistent_safety}
The learned safety filter policy \eqref{eq:learned_qp} is persistently safe 
for all initial states $x_0$ within a safe set $\mathcal{X}_0 = \{x \mid Gx \le c\}$ 
and for any given reference control $\hat{u}$, if the optimal value $z^\star$ of 
the following nonconvex Quadratically Constrained Quadratic Program (QCQP) is nonnegative:
\begin{subequations}
\label{eq:qcqp_safety_certificate}
\begin{align}
\min_{x_0, \hat{u}, v, y, \mu} \quad & - v^\top(G(Ax_0 + B y_{[1:m_{\text{sys}}]}) - c) \\
\text{s.t.} \quad & Gx_0 \le c, \label{eq:qcqp_initial_safe} \\
& v \ge 0, \quad \mathbf{1}^\top v = 1, \label{eq:qcqp_v_constraints} \\
& Py - E_1^\top \hat{u} - H^\top \mu = 0, \label{eq:qcqp_kkt_stationarity} \\
& Hy + W_b x_0 + b_b \ge 0, \label{eq:qcqp_kkt_primal_feas} \\
& \mu \ge 0, \label{eq:qcqp_kkt_dual_feas} \\
& \mu^\top (Hy + W_b x_0 + b_b) = 0. \label{eq:qcqp_kkt_complementary}
\end{align}
\end{subequations}
\end{theorem}

\begin{pf}
The safety filter is persistently safe if, 
for any initial state $x_0 \in \mathcal{X}_0$, 
the resulting next state $x_1 = Ax_0 + B u_0$ also remains in the safe set, 
i.e., $x_1 \in \mathcal{X}_0$. This requires the condition 
$G(Ax_0 + B\pi_\theta(x_0, \hat{u})) \le c$ to hold for all $x_0$ satisfying $Gx_0 \le c$.

This universal condition is equivalent to stating that 
the maximum possible violation is non-positive:
\begin{equation}
\max_{x_0: Gx_0 \le c} \left[ \max_{i} \left( G_i(Ax_0 + B\pi_\theta(x_0, \hat{u})) - c_i \right) \right] \le 0,
\end{equation}
where $G_i$ and $c_i$ are the $i$-th rows of $G$ and $c$, respectively.

By introducing an auxiliary vector $v \ge 0$ with $\mathbf{1}^\top v = 1$, 
we can express the inner maximum as a linear program. 
Leveraging the strong duality of linear programming, the condition becomes:
\begin{equation}
\max_{x_0: Gx_0 \le c} \left[ \max_{v \ge 0, \mathbf{1}^\top v = 1} v^\top(G(Ax_0 + B\pi_\theta(x_0, \hat{u})) - c) \right] \le 0.
\end{equation}
Combining two maximization and flipping it to a minimization, 
we obtain the equivalent condition that the optimal value of 
the following problem must be non-negative:
\begin{equation}
\min_{x_0, \hat{u}, v} - v^\top(G(Ax_0 + B\pi_\theta(x_0, \hat{u})) - c) \ge 0,
\end{equation}
subject to the constraints in \eqref{eq:qcqp_initial_safe} and \eqref{eq:qcqp_v_constraints}.

The expression $\pi_\theta(x_0, \hat{u})$ represents 
the inner-level optimization problem from \eqref{eq:learned_qp}. 
This creates a bilevel optimization problem, 
which is notoriously difficult to solve. 
The key step is to replace the inner optimization problem 
with its Karush-Kuhn-Tucker (KKT) conditions, 
which are necessary and sufficient for optimality for a convex QP. 
The KKT conditions for the QP in \eqref{eq:learned_qp} 
are precisely the constraints \eqref{eq:qcqp_kkt_stationarity}-\eqref{eq:qcqp_kkt_complementary}, 
where $\mu$ is the vector of Lagrange multipliers.
By performing this substitution, we transform the bilevel problem 
into the single-level, albeit nonconvex, QCQP as stated in the theorem. 

\end{pf}

\textit{Numerical verification via SOS/SDP:} 
The problem \eqref{eq:qcqp_safety_certificate} is a polynomial (at most quadratic) 
optimization in the variables $(x_0,\hat{u},v,y,\mu)$ with complementarity conditions
\eqref{eq:qcqp_kkt_complementary} producing bilinear terms. 
Globally solving this nonconvex QCQP is, in general, hard. 
A practical and certifiable approach is to apply Sum-of-Squares (SOS) programming \citep{Parrilo2003} 
to relax this nonconvex QCQP problem to a Semi-definite Programming (SDP) problem, 
and solve the SDP problem to obtain a lower bound $z_{\mathrm{SDP}}\le z^\star$ \citep{Lasserre2001}.
If the computed lower bound satisfies $z_{\mathrm{SDP}}\ge 0$, 
then $z^\star\ge 0$ and Theorem \ref{thm:persistent_safety} yields the certificate, 
i.e., the learned safety filter is persistently safe for all states within $\mathcal{X}_0$ 
and for any reference control $\hat{u}$.

\section{Numerical Results}
\begin{table*}[ht]
  \centering
  \caption{Performance comparison on cartpole system.\protect\footnotemark}
  \label{tab:performance_table_cartpole}
  \resizebox{\textwidth}{!}{%
  \begin{tabular}{>{\centering\arraybackslash}m{2.8cm}|cccccccc|cccccccc}
    \toprule
     & \multicolumn{8}{c|}{Stabilization Tasks} & \multicolumn{8}{c}{Tracking Tasks} \\ 
    \cmidrule(lr){2-9} \cmidrule(lr){10-17} 
    Method & \multicolumn{2}{c}{$n=0.0$} & \multicolumn{2}{c}{$n=1.0$} & \multicolumn{2}{c}{$n=5.0$} & \multicolumn{2}{c|}{$n=8.0$} & \multicolumn{2}{c}{$n=0.0$} & \multicolumn{2}{c}{$n=1.0$} & \multicolumn{2}{c}{$n=5.0$} & \multicolumn{2}{c}{$n=8.0$} \\ 
    \cmidrule(lr){2-3} \cmidrule(lr){4-5} \cmidrule(lr){6-7} \cmidrule(lr){8-9} 
    \cmidrule(lr){10-11} \cmidrule(lr){12-13} \cmidrule(lr){14-15} \cmidrule(lr){16-17}
     & Vio\% & Dev & Vio\% & Dev & Vio\% & Dev & Vio\% & Dev & Vio\% & Dev & Vio\% & Dev & Vio\% & Dev & Vio\% & Dev \\ 
    \midrule
    Unfiltered     & 0.00 & - & 0.00 & - & 78.67 & - & 85.31 & - & 0.00 & - & 0.00 & - & 23.31 & - & 97.33 & - \\
    \cmidrule(lr){1-17}
    Baseline PSF   & 0.00 & $\bm{6.89 \times 10^{-5}}$   & 0.00 & $\bm{1.90 \times 10^{-3}}$ & 0.00 & \underline{174.76} & fail & -           & 0.00 & $\bm{8.13 \times 10^{-5}}$ & 0.00 & $\bm{8.53 \times 10^{-4}}$ & 68.08 & 693.18 & fail & - \\
    RL+MLP         & 0.00 & 1.27                          & 0.00 & 10.51 & 7.37 & 258.41              & 9.56 & 406.55     & 0.00 & 30.21                  & 0.00 & 35.81          & 0.00 & \underline{1220.42} & 0.67 & \underline{2000.97} \\
    RL+LQP         & 0.00 & \underline{$7.46 \times 10^{-2}$} & 0.00 & \underline{0.15}             & 0.00 & \textbf{164.10} & 0.00 & \textbf{402.90} & 0.00 & \underline{$6.51 \times 10^{-3}$} & 0.00 & \underline{$3.00 \times 10^{-2}$} & 0.00 & \textbf{189.25} & 0.00 & \textbf{767.45} \\ 
    \bottomrule
  \end{tabular}%
  }
\end{table*}

\begin{figure*}[h]
   \centering
   \includegraphics[width=16.4cm]{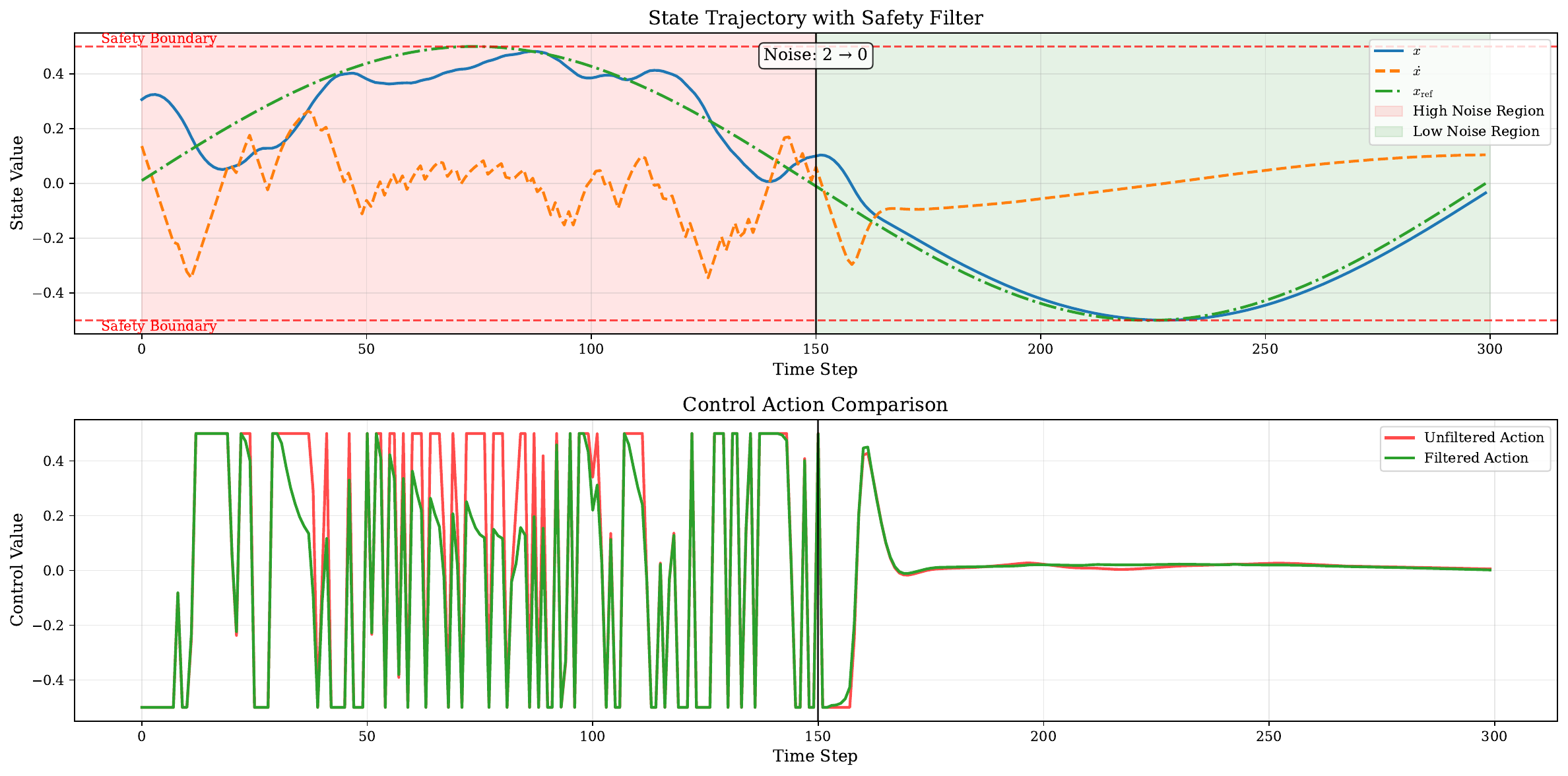}
   \caption{Performance of the proposed RL+LQP safety filter on a double integrator tracking task 
   under high-noise ($n=2$, red background) and low-noise ($n=0$, green background) regimes. 
   Top: state trajectories remain within the safety bounds (red dashed lines) in both regimes. 
   Bottom: corresponding control actions show increased intervention in high-noise regions and near-perfect tracking in low-noise regions.}
   \label{fig:track_traj}
\end{figure*}

This section presents experimental validation of the proposed model-free safety filter 
through simulations across three benchmark systems: 
double integrator, quadruple tank \citep{Johansson2000}, and cartpole \citep{Geva1993}. 
Experiments evaluate performance on both \textit{stabilization tasks} (fixed setpoints) and 
\textit{tracking tasks} (time-varying references). 
Comparative analysis includes:
\begin{itemize}
    \item \textit{Unfiltered Control}: apply reference input $\hat{u}$ directly.
    \item \textit{Baseline PSF}: original model-based predictive safety filter \citep{Wabersich2021}.
    \item \textit{RL+MLP}: an RL-trained MLP safety filter.
    \item \textit{RL+LQP (Proposed)}: our RL-trained learnable QP safety filter.
\end{itemize}

Experimental setup: the candidate control $\hat{u}$ is generated by Linear Quadratic Regulator (LQR) controller
with additive Gaussian noise of different mean values $n$. 
Each configuration was evaluated over 100 independent episodes. 
Evaluation metrics include \textit{Violation Rate (Vio\%)} (The percentage of steps 
violating safety constraints out of the total number of control steps) and 
\textit{Cumulative Deviation} ($Dev = \sum \| \hat{u} - u_0 \|^2 $).
Additionally, all given LQP results is assessed in appropriate configurations: $n_{\text{qp}}=4, m_{\text{qp}}=30$.
And the detailed models and configurations of each system are provided in Appendix \ref{sec:appendix}.

Table \ref{tab:performance_table_cartpole} provides a comprehensive comparison of safety filter performance 
on the cartpole system under varying noise conditions. The evaluation covers both stabilization 
and tracking tasks. 
For intuitive demonstration, best is highlighted in bold, while second best is underlined.

\subsection{Stabilization Tasks}
For stabilization tasks, according to the left half of Table \ref{tab:performance_table_cartpole},
when the LQR control noise is absent ($n=0.0$), the system can operate without filtering 
and still comply with safety constraints.
However, as noise $n$ increases, unfiltered control inputs will result in 
increasingly severe violations of safety constraints. 
While after filtering with our proposed RL+LQP method, 
empirically 100\% safety guarantee can be achieved for any given noise level.

For example, when the noise level $n$ increased to 8.0, the model-based PSF baseline failed
due to the unsolvable error of the OSQP solver, thus losing its filtering ability.
Meanwhile, though RL+MLP method has reduced Vio\% to some extent, 
there are still 9.56\% of time steps that are unsafe.
At this point, only our proposed RL+LQP method 
did not exhibit any observed violations across all tested conditions.

\subsection{Tracking Tasks}
For tracking tasks, the reference control $\hat{u}$ is set as a sinusoidal trajectory, 
whose amplitude is equal to the safety boundary. 
Consequently, the closer the trajectory is to its peak, the higher the safety risk.

The right half of Table \ref{tab:performance_table_cartpole} compares tracking performance for the cartpole system.
Our RL+LQP method consistently outperformed baselines in safety-critical scenarios 
with superior safety guarantees while maintaining lower deviation.

Figure \ref{fig:track_traj} illustrates the key capabilities of our proposed safety filter in double integrator system. 
The state trajectories (top) show that the filter successfully maintains all states 
within the safety boundaries (red dashed lines) across both high-noise ($n=2$, red background) and low-noise ($n=0$, green background) regimes.

The control actions (bottom) reveal its intelligent intervention strategy: 
\textit{In high-noise regions}, the filter intervenes substantially (green vs. red line), 
particularly as the state $x$ approaches the constraint limits. 
Notably, it demonstrates a \textit{proactive} safety-assurance behavior, 
initiating corrective actions \textit{before} violations occur. 
\textit{In low-noise regions}, where the reference input is inherently safe, 
the filter tracks with near-perfect fidelity and minimal intervention, 
prioritizing control accuracy when safety is not compromised.

This behavior, characterized by proactive intervention only when safety is at risk 
and high-fidelity tracking otherwise, 
demonstrates that our learned filter effectively balances safety assurance with minimal modification.

The experimental results on other systems are similar.
Due to space constraints, we only present a subset of the results in the main paper. 
The complete results, including statistical tables of other systems (double integrator and quadruple tank), 
and detailed system configurations, are provided in Appendix \ref{sec:appendix}.

\textit{Summary of Properties:} As shown in the experimental results, in both stabilization and tracking tasks, 
the proposed RL+LQP safety filter demonstrates three key advantages over existing approaches:

\begin{enumerate}
    \item \textit{Enhanced Safety Guarantees}: Our approach outperforms both baselines with great safety performance, 
    achieving empirically 0\% violation rate even in challenging high-noise conditions where the conventional model-based PSF fails to find feasible solutions. 
    This comparison arises because conventional solvers, operating with a limited prediction horizon, prioritize short-term constraint satisfaction and minimal intervention, 
    resulting in \textit{myopic} behavior that may progressively drive the system into unsalvageable dangerous states.

    \item \textit{Minimal Intervention}: 
    When the noise level is low, both our method and the conventional PSF exhibit negligible intervention, 
    while under high noise, our method excels. 
    This also stems from the far-sighted perspective of our learning-based method, 
    which optimizes the safety filter's policy over entire episodes rather than a finite horizon,  
    achieving a balance between immediate safety and minimal long-term intervention.
    Moreover, across all tested noise conditions, our approach achieves much lower control deviation than the RL+MLP baseline, 
    demonstrating the efficacy of the embedded QP structure.

    \item \textit{Lower Computational Load}: As quantified in Table \ref{tab:computational_efficiency}, 
    considering FLOPs (floating point operations per control step), 
    our RL+LQP approach requires substantially fewer online computational resources than alternatives.
\end{enumerate}
\begin{table}[h]
    \centering
    \caption{\protect\mbox{Computational load comparison (FLOPs).}}
    \label{tab:computational_efficiency}
    \resizebox{8.9cm}{!}{
    \begin{tabular}{c|ccc}
    \toprule
    \diagbox{Method}{System} & Double Integrator & Cartpole & Tank \\
    \midrule
    Baseline PSF & $1.04 \times 10^{6}$ & $2.17 \times 10^{6}$ & $2.57 \times 10^{6}$ \\
    RL+MLP & $8.45 \times 10^{4}$ & $8.66 \times 10^{4}$ & $8.67 \times 10^{4}$ \\
    RL+LQP & $\bm{2.12 \times 10^{4}}$ & $\bm{5.99 \times 10^{4}}$ & $\bm{2.15 \times 10^{4}}$ \\
    \bottomrule
    \end{tabular}
    }
 \end{table}

Such advantages originate from the unrolled QP network's capability 
to encode temporal safety constraints in a model-free manner, 
enabling a compromise between safety assurance and computational practicality.

\footnotetext{Entries labeled ``fail'' denote cases where the OSQP solver returned infeasible.}

\section{Conclusion}
\label{sec:conclusion}
This paper proposed a model-free predictive safety filter 
that parameterizes the safety-check QP as an unrolled solver 
and learns its parameters via reinforcement learning. 
This method preserves the interpretable QP structure, which enables formal safety verification, 
while removing reliance on explicit system models. 
We also provided a sufficient certificate for its persistent safety, 
bridging the gap between learning-based methods and formal guarantees. 
Experiments on benchmark tasks demonstrate 
strong empirical safety guarantee and reduced intervention compared to baselines, 
and substantially lower per-step computation than conventional model-based PSFs. 

Future work will focus on extending the safety certificate to handle system uncertainties 
and practical deployment on real-world platforms. 
Code is available at https://github.com/bbihui/LQP-PSF. 
The open-source implementation will facilitate further research 
in learning-based safety-critical control. 
Subsequent investigations could incorporate robust control formulations 
to enhance performance under unmodeled dynamics and extreme operating conditions. 

\begin{table*}[htbp]
  \centering
  \caption{Task definitions for benchmarking.}
  \label{tab:task_definitions}
  \begin{tabular}{l l l l}
    \toprule
    & \text{Double Integrator} & \text{Quadruple Tank} & \text{Cartpole Balance} \\
    \midrule
    \text{Source} & 
    \makecell[l]{\cite{Borrelli2017}, \\ p.246, Example 12.1} & 
    \makecell[l]{\cite{Johansson2000} \\ (original source); \\ \cite{Zishuo2023} \\ (linearization)} & 
    \makecell[l]{\cite{Geva1993}} \\
    \midrule
    \text{Nominal dynamics} & 
    $\begin{array}{l} A = \begin{bmatrix} 1 & 1 \\ 0 & 1 \end{bmatrix}, \\ B = \begin{bmatrix} 0 \\ 1 \end{bmatrix} \end{array}$ &
    $\begin{array}{l} A = \begin{bmatrix} 0.98 & 0 & 0.04 & 0 \\ 0 & 0.99 & 0 & 0.03 \\ 0 & 0 & 0.96 & 0 \\ 0 & 0 & 0 & 0.97 \end{bmatrix}, \\ B = \begin{bmatrix} 0.83 & 0 \\ 0 & 0.62 \\ 0 & 0.47 \\ 0.3 & 0 \end{bmatrix} \end{array}$ &
    $\begin{array}{@{}l@{}}
    \begin{bmatrix}
        m_c + m_p & m_p l \cos\theta \\
        m_p l \cos\theta & m_p l^2
    \end{bmatrix}
    \begin{bmatrix}
        \ddot{p}_x \\ \ddot{\theta}
    \end{bmatrix} 
    = 
    \begin{bmatrix}
        u + m_p l \sin\theta \dot{\theta}^2 \\
        m_p g l \sin\theta
    \end{bmatrix}, \\
    \text{where state variable is } x = [p_x, \dot{p}_x, \theta, \dot{\theta}]^\top, \\
    \text{and nominal parameters are } m_c=1,
    m_p=0.1, l=0.55. \\
    \text{Discretized with time step 0.1s.} \\
    \text{Simulation uses nonlinear dynamics, } \\
    \text{while baseline PSF uses linearized dynamics.}
    \end{array}$ \\
    \midrule
    \text{Safety Constraints} & 
    $\begin{array}{l} -0.5 \le x \le 0.5, \\ -0.5 \le u \le 0.5 \end{array}$ & 
    $\begin{array}{l} 0 \le x \le 20, \\ -1 \le u \le 1 \end{array}$ & 
    $\begin{array}{l} -2 \le p_x \le 2, \\ -0.5 \le \theta \le 0.5, \\ -10 \le u \le 10 \end{array}$ \\
    \midrule
    \text{Episode length} &
    100 &
    100 &
    300 \\
    \bottomrule
  \end{tabular}
\end{table*}

\begin{table*}[ht] 
\centering
\caption{Performance comparison on double integrator system.} 
\label{tab:performance_table_double_integrator}
\resizebox{\textwidth}{!}{%
\begin{tabular}{c|cccccccc|cccccccc}
\toprule 
& \multicolumn{8}{c|}{Stabilization Tasks} & \multicolumn{8}{c}{Tracking Tasks} \\ 
\cmidrule(lr){2-9} \cmidrule(lr){10-17} 
Method & \multicolumn{2}{c}{$n=0.0$} & \multicolumn{2}{c}{$n=0.5$} & \multicolumn{2}{c}{$n=1.0$} & \multicolumn{2}{c|}{$n=2.0$} & \multicolumn{2}{c}{$n=0.0$} & \multicolumn{2}{c}{$n=0.5$} & \multicolumn{2}{c}{$n=1.0$} & \multicolumn{2}{c}{$n=2.0$} \\ 
\cmidrule(lr){2-3} \cmidrule(lr){4-5} \cmidrule(lr){6-7} \cmidrule(lr){8-9} 
\cmidrule(lr){10-11} \cmidrule(lr){12-13} \cmidrule(lr){14-15} \cmidrule(lr){16-17}
     & Vio\% & Dev & Vio\% & Dev & Vio\% & Dev & Vio\% & Dev & Vio\% & Dev & Vio\% & Dev & Vio\% & Dev & Vio\% & Dev \\ 
\midrule 
Unfiltered  & 0.00 & - & 8.67 & - & 36.12 & - & 81.34 & - & 0.00 & - & 0.00 & - & 12.67 & - & 16.54 & - \\
Baseline PSF    & 0.00 & $1.16 \times 10^{-7}$ & 0.00 & $2.79 \times 10^{-7}$ & 0.00 & 4.67 & fail & - & 0.00 & $1.41 \times 10^{-2}$ & 0.00 & $5.87 \times 10^{-2}$ & 0.00 & 2.09 & fail & - \\
RL+MLP      & 0.00 & 0.23 & 0.00 & 0.63 & 6.33 & 2.32 & 17.67 & 2.88 & 0.00 & 0.11 & 0.00 & 0.24 & 0.00 & 7.41 & 0.00 & 9.12 \\
RL+LQP         & 0.00 & 0.49 & 0.00 & 0.68 & 0.00 & 2.73 & 0.00 & 4.50 & 0.00 & 0.13 & 0.00 & 0.89 & 0.00 & 5.05 & 0.00 & 5.03 \\ 
\bottomrule 
\end{tabular}%
}
\end{table*}

\begin{table*}[ht]
  \centering
  \caption{Performance comparison on quadruple tank system in stabilization tasks.}
  \label{tab:performance_table_tank}
  \resizebox{\textwidth}{!}{%
  \begin{tabular}{c|cccccccc|cccc}
    \toprule
     & \multicolumn{8}{c|}{LQR control with Gaussian noise} & \multicolumn{4}{c}{Bang-bang control} \\ 
    \cmidrule(lr){2-9} \cmidrule(lr){10-13} 
    Method & \multicolumn{2}{c}{$n=0.0$} & \multicolumn{2}{c}{$n=0.2$} & \multicolumn{2}{c}{$n=1.0$} & \multicolumn{2}{c|}{$n=5.0$} & \multicolumn{2}{c}{$(0.3, 30)$} & \multicolumn{2}{c}{$(1, 50)$} \\ 
    \cmidrule(lr){2-3} \cmidrule(lr){4-5} \cmidrule(lr){6-7} \cmidrule(lr){8-9} 
    \cmidrule(lr){10-11} \cmidrule(lr){12-13} 
     & Vio\% & Dev & Vio\% & Dev & Vio\% & Dev & Vio\% & Dev & Vio\% & Dev & Vio\% & Dev \\ 
    \midrule
    Unfiltered  & 0.00 & - & 2.35 & - & 41.02 & - & 49.14 & - & 0.00 & - & 93.63 & -  \\
    Baseline PSF   & 0.00 & $64.63 \times 10^{-19}$ & 0.00 & $2.26 \times 10^{-19}$ & 0.00 & 3.96 & 0.00 & 9.26 & 0.00 & $1.96 \times 10^{-29}$ & 0.00 & 62.01 \\
    RL+MLP      & 0.00 & 3.22 & 0.00 & 5.38 & 0.00 & 38.39 & 0.00 & 54.85 & 0.00 & 3.53 & 0.00 & 55.62 \\
    RL+LQP      & 0.00 & $1.83 \times 10^{-2}$ & 0.00 & $2.89 \times 10^{-2}$ & 0.00 & 15.47 & 0.00 & 25.27 & 0.00 & 0.38 & 0.00 & 60.73 \\ 
    \bottomrule
  \end{tabular}%
  }
\end{table*}

\bibliography{ifacconf}             

\appendix
\section{benchmarking setup and additional experiment results}    
\label{sec:appendix}
In this appendix, we provide the detailed system configurations and 
additional numerical results on double integrator and quadruple tank systems.
The task definitions for the three benchmark systems are summarized in Table \ref{tab:task_definitions}.
Table \ref{tab:performance_table_double_integrator} and \ref{tab:performance_table_tank} 
summarize key results in double integrator system and quadruple tank, respectively.

Notably, for the quadruple tank system, we only conducted stabilization tasks.
But in addition to the LQR controller with Gaussian noise, 
an open-loop bang-bang control reference $\hat{u}$ was also applied for evaluation. 
This controller switches between a non-zero control input and zero, defined as:
\begin{equation}
\label{eq:bang_bang_control}
\hat{u}(t) = 
\begin{cases} 
u & \text{if } t < t_{\text{switch}} \\
0 & \text{if } t \ge t_{\text{switch}}
\end{cases}
\end{equation}
In Table \ref{tab:performance_table_tank}, this control strategy is denoted by the tuple $(u, t_{\text{switch}})$, 
representing the control magnitude and the switching time step.
                                                     
\end{document}